\documentclass[11pt]{amsart}
\usepackage{amsmath, latexsym, amssymb}
\usepackage[parfill]{parskip} 
\theoremstyle{plain}
\newtheorem{lemma}{Lemma}[section]
\newtheorem{proposition}[lemma]{Proposition}
\newtheorem{theorem}[lemma]{Theorem}

\theoremstyle{definition}

\newtheorem{remark}[lemma]{Remark}

\begin{document}
\title[Parameterized Gromov-Witten Invariants]{\large  Parameterized Gromov-Witten invariants
and topology of symplectomorphism groups}
\author[H.-V. Le and K. Ono]{H\^ong-V\^an L\^e$^{1}$
and Kaoru Ono$^{2} $} \thanks{ $^1$
Partially supported by grant of ASCR Nr IAA100190701,\\
$^2$
Partially supported by Grant-in-Aid for Scientific Research 
Nos. 08640093 and 18340014, Ministry of Education, Culture, Sports, 
Science and Technology, Japan}

\maketitle

\newcommand{\R}{{\bf R}}
\newcommand{\C}{{\bf C}}
\newcommand{\F}{{\bf F}}
\newcommand{\Z}{{\bf Z}}
\newcommand{\N}{{\bf N}}
\newcommand{\Q}{{\bf Q}}
\newcommand{\Aa}{{\mathcal A}}  
\newcommand{\Bb}{{\mathcal B}}
\newcommand{\Cc}{{\mathcal C}}    
\newcommand{\Dd}{{\mathcal D}}
\newcommand{\Ee}{{\mathcal E}}
\newcommand{\Ff}{{\mathcal F}}
\newcommand{\Gg}{{\mathcal G}}    
\newcommand{\Hh}{{\mathcal H}}
\newcommand{\Kk}{{\mathcal K}}
\newcommand{\Jj}{{\mathcal J}}
\newcommand{\Ll}{{\mathcal L}}    
\newcommand{\Mm}{{\mathcal M}}  
\newcommand{\Nn}{{\mathcal N}}
\newcommand{\Oo}{{\mathcal O}}
\newcommand{\Pp}{{\mathcal P}}
\newcommand{\Qq}{{\mathcal Q}}
\newcommand{\Rr}{{\mathcal R}}
\newcommand{\Ss}{{\mathcal S}}
\newcommand{\Tt}{{\mathcal T}}
\newcommand{\Uu}{{\mathcal U}}
\newcommand{\Vv}{{\mathcal V}}
\newcommand{\Ww}{{\mathcal W}}
\newcommand{\Xx}{{\mathcal X}}
\newcommand{\Yy}{{\mathcal Y}}
\newcommand{\Zz}{{\mathcal Z}}
\newcommand{\zt}{{\tilde z}}
\newcommand{\xt}{{\tilde x}}
\newcommand{\Ht}{\widetilde{H}}
\newcommand{\ut}{{\tilde u}}
\newcommand{\Mt}{{\widetilde M}}
\newcommand{\Llt}{{\widetilde{\mathcal L}}}
\newcommand{\yt}{{\tilde y}}
\newcommand{\vt}{{\tilde v}}
\newcommand{\Ppt}{{\widetilde{\mathcal P}}}
\newcommand{\Remark}{{\it Remark}}
\newcommand{\Proof}{{\it Proof}}
\newcommand{\ad}{{\rm ad}}
\newcommand{\Om}{{\Omega}}
\newcommand{\om}{{\omega}}
\newcommand{\eps}{{\varepsilon}}
\newcommand{\Di}{{\rm Diff}}
\newcommand{\Pro}[1]{\noindent {\bf Proposition #1}}
\newcommand{\Thm}[1]{\noindent {\bf Theorem #1}}
\newcommand{\Lem}[1]{\noindent {\bf Lemma #1 }}
\newcommand{\An}[1]{\noindent {\bf Anmerkung #1}}
\newcommand{\Kor}[1]{\noindent {\bf Korollar #1}}
\newcommand{\Satz}[1]{\noindent {\bf Satz #1}}
\newcommand{\gl}{{\frak gl}}
\newcommand{\so}{{\frak so}}
\newcommand{\su}{{\frak su}}
\newcommand{\ssl}{{\frak sl}}
\newcommand{\ssp}{{\frak sp}}
\newcommand{\g}{{\frak g}}    
\newcommand{\Cinf}{C^{\infty}}
\newcommand{\CS}{{\mathcal{CS}}}
\newcommand{\YM}{{\mathcal{YM}}}
\newcommand{\Jreg}{{\mathcal J}_{\rm reg}}
\newcommand{\Hreg}{{\mathcal H}_{\rm reg}}
\newcommand{\SP}{{\rm SP}}
\newcommand{\im}{{\rm im}}
\newcommand{\inner}[2]{\langle #1, #2\rangle}    
\newcommand{\half}{\scriptstyle\frac{1}{2}}
\newcommand{\p}{{\partial}}
\newcommand{\notsub}{\not\subset}
\newcommand{\iI}{{I}}           
\newcommand{\bI}{{\partial I}}     
\newcommand{\LRA}{\Longrightarrow}
\newcommand{\LLA}{\Longleftarrow}
\newcommand{\lra}{\longrightarrow}
\newcommand{\LLR}{\Longleftrightarrow}
\newcommand{\lla}{\longleftarrow}
\newcommand{\INTO}{\hookrightarrow}
\newcommand{\Sy}{\text{ Diff }_{\om}}
\newcommand{\Ex}{\text{Diff }_{ex}}
\newcommand{\jdef}[1]{{\bf #1}}
\newcommand{\QED}{\hfill$\Box$\medskip}
\newcommand{\dass}{da\ss~}
\newcommand{\UuU}{\Upsilon _{\delta}(H_0) \times \Uu _{\delta} (J_0)}

\begin{abstract}
In this note we introduce parameterized Gromov-Witten invariants for 
symplectic fiber bundles 
and study the topology of the symplectomorphism group.  
We also give sample applications showing the non-triviality of certain
homotopy groups of some symplectomorphism groups.
\end{abstract}

\bigskip


\section{Introduction}

Given a symplectic  manifold $(M,\om)$, one of the basic  mathematical objects associated 
to $(M,\om)$ is 
its automorphism group ${\rm Symp} (M,\om)$.  
Since the group ${\rm Symp} (M,\om)$ can be equipped
with the $C^{\infty}$ topology, we would 
like  to know the homotopy type of this  automorphism group. In this note we are interested
in the following questions:

 1) How large is the
rank of the homotopy group $\pi_i ({\rm Symp} (M,\om) \otimes \Q)$?

 2) What are the characteristic classes
of ${\rm Symp} (M,\om)$, that is, the cohomology ring
of the classifying space ${\rm BSymp} (M,\om)$.

Our approach to these problems uses symplectic fiber bundle setting 
(a similar setting is used in the study of homotopy
type of diffeomorphism groups) and Gromov's technique of pseudoholomorphic curves. We would also like to remark that the Gromov technique of pseudoholomorphic curves 
has been developed and
extended in different directions in the study of the topology of  symplectomorphism groups.  
Recent developments in this direction can be found in, e.g., McDuff's 
survey \cite{McDuff2004}.

This note consists of four sections.  
In  section 2  we recall the definition of  symplectic fiber 
bundles and we introduce the notion of fiber-wise (vertical)
stable maps.  
In section 3 we show that
the  basic properties of the moduli space of stable maps also
hold in the fiber-wise (family) version. As an immediate consequence
we construct parameterized Gromov-Witten invariants for symplectic fiber bundles, 
which is a family version of the usual Gromov-Witten invariants for symplectic 
manifolds.  
In  section 4 we apply this construction 
to  problems 1, 2 mentioned above.  
We associate to each element $\pi _i ({\rm Symp} (M,\om))$ 
a symplectic fiber bundle over $S^{i+1}$
which is the union of two trivial symplectic
bundles  over a disk $D^{i+1}$ glued along the boundary
$\p D^{i+1}$ by this element $\pi_i ({\rm Symp} (M,\om))$. 
We re-interpret a result by 
Gromov \cite{Gromov1985}, Theorem 2.4.C$_2$ on the existence of an element of
infinite oder in the symplectomorphism group of non-monotone $S^ 2\times S^ 2$ 
in terms of parameterized Gromov-Witten invariants, 
see Theorem \ref{Theorem 4.3}. We also slightly generalize 
Gromov's result in the following cases.
We denote by $(X^4_1,\om_1)$ a non-monotone symplectic manifold which is
diffeomorphic to $S^2\times S^2$, and by $(X^4_2,\om_2)$ a symplectic
manifold which is diffeomorphic to $\C P^2\# \overline{\C P^2}$.

\medskip

{\bf Theorem \ref{Theorem 4.4}.} {\it a)  Let $(M_1,\Om_1)= (X^4_1\times N^{2k},\om_1\oplus \om_0)$ be a symplectic manifold 
 with $(X^4_1,\om_1)$ as above and $(N,\om_0)$  a compact symplectic 
manifold.
Then we have $rk (\pi_1 ({\rm Symp} (M_1,\Om_1))\otimes \Q)\ge 1$.

b) We also have $rk (\pi_1({\rm Symp} (X^4_2, \om_2)) \ge 1$.}

\medskip

There are intensive studies on cohomology groups and homotopy types of 
symplectomorphism groups of rationally ruled symplectic 4-manifolds such as 
Abreu \cite{Abreu1998}, Abreu-McDuff \cite{A-M2000}, Anjos 
\cite{An}, etc.  
In fact, we can improve Theorem 4.5 
for the case of $(M_1,\Om_1)$ 
without using ``hard machinery'.

{\bf Theorem  \ref{Theorem 4.7}.}{\it a) The rank of the homomorphism 
$i_*:\pi_1({\rm Symp} (M_1,\Om_1)) \to \pi_1({\rm Diff}(M_1))$ is at least 1.

b) The rank of the homomorphism $i_* :\pi_3({\rm Symp} (M_1,\Om_1)) 
\to \pi_3({\rm Diff}(M_1))$ is greater than or equal 2.}

In   section 4 we also construct characteristic classes of
the group ${\rm Symp} (M,\om)$ by formulating the Gromov-Witten 
invariants in a dual way.
We also include an Appendix, which contains an alternative proof of 
Theorem \ref{Theorem 4.3}, a special version of Theorem \ref{Theorem 4.4}.a.

After the preliminary version of this note was written \cite{LO2001}, 
we learned several works on the topology of symplectomorphism groups, 
\cite{Kedra}, \cite{Nishinou}, see also references in \cite{McDuff2004}. 
Since some results in \cite{LO2001} have been quoted in some literature 
e.g.  \cite{Buse2005}, \cite{Busepreprint}, \cite{McDuff2004}, we feel a need 
to revise the version \cite{LO2001} 
to correct some errors as well as to add details to missing arguments. 

\medskip

{\bf Acknowledgement.} 
We thank Forschungsinstitut Oberwolfach, where we exchanged 
some ideas during our stay in RIP program 1996. 
The first author is indebted to John Lott
for his inspiring lecture \cite{Lott1996} and
his help in the reference in the homotopy type of diffeomorphism groups.
We are also grateful to Dusa McDuff for her interest and
critical  helpful comments.  A part of this note was   written 
during the stay of the first author at the Max-Planck-Institute in Bonn
and the Mathematical Institute of the University Leipzig during her Heisenberg 
fellowship. She thanks all these institutions for their hospitality.
The second author thanks Professor Akira Kono, who suggested 
the way of simplifying the proof of Theorem \ref{Theorem 4.7}.

\section{Symplectic fiber bundles and vertical stable maps}

In this section we recall the notions of symplectic fiber bundles, 
stable maps and introduce the notion of vertical stable maps. 
We refer to \cite{G-L-S1996} 
for more discussions on symplectic fiber bundles. 
The idea of counting fiber-wise holomorphic curves in 
symplectic fiber bundles 
is also suggested by Kontsevich (in his communication to us after a preliminary version of this note has been written in 1997) and by Lu-Tian\footnote{We thank Dusa McDuff for 
informing us that they used this idea in order to construct equivariant Gromov-Witten invariants, which is 
a special case of the parameterized Gromov-Witten invariants for 
the symplectic fiber bundle associated with the Hamiltonian action 
of a compact Lie group on a symplectic manifold.}.  

\subsection{Symplectic bundles and their fiber-wise compatible
almost complex structures}

A fibration $M\to E \stackrel{\pi}{\to }B$ is said to be a symplectic fiber bundle,
if the fiber $M$ is diffeomorphic to a symplectic  manifold $(M,\om)$
and the transition function takes its value in
the group ${\rm Symp} (M,\om)$. 
We denote by $\Jj_{\pi}(E)$ the associated bundle over $B$ 
whose fiber is the space $\Jj(M)$ 
of smooth compatible  almost complex structures 
on $(M,\om)$.  
Since the fiber $\Jj(M)$ is contractible, the space of sections 
$J(E):B \to \Jj_{\pi}(E)$ is also non-empty and contractible. 

In what follows, we are interested in defining invariants which
detect the non-triviality of symplectic fiber bundles.  
Associated to any symplectic fiber bundle 
$M \to E\stackrel{\pi}{\to} B$ we obtain the local system 
of fiberwise homology groups, resp. fiberwise cohomology groups, denoted by 
$\Hh_*(E)$, resp. $\Hh^*(E)$.  In what follows,
the coefficients of cohomology groups are in $\R$ or $\Z$,
and the coefficients of homology groups are in $\Z$.   
Clearly all the invariants of the associate local systems $\Hh_*(E), \Hh^*(E)$ are 
also invariants of symplectic fiber bundles.  
In particular, if a symplectic fiber bundle is trivial, 
then the associated local systems are simple, i.e., trivial. 
 We also observe that  for
a symplectic fiber bundle $E$ there is 
always a section $s^{[\om]}: B \to \Hh^2(E)$ which
takes a given value $[\om]\in H^2(M,\R)$ and
 there is also a section $s^{c_1}: B \to \Hh^2(E)$
which takes a given value $c_1(M,\om)\in H^2(M,\Z)$.

\subsection{Stable maps and vertical stable maps.}  

Our notion of vertical stable maps is  based on  the notion of stable maps due to Kontsevich \cite{K-M1994}, \cite{Kontsevich1995}, see also \cite{F-O1999} whose exposition we follow closely.

Let $g$ and $m$ be nonnegative integers. A semistable curve with $m$ marked  points is 
a pair $(\Sigma, z)$ of a connected space  $\Sigma = \cup \pi_{\nu} (C_\nu)$, 
where $C_\nu$ is a Riemann surface and $\pi _{\nu}:C_\nu \to \Sigma$ is 
a continuous map, and $ z = ( z_1, \cdots , z_m)$ are $m$ distinct points 
in $\Sigma$ with the following properties.\\
(1) \  $\pi_{\nu}$ is the normalization of the irreducible component 
$\Sigma_{\nu}=\pi_{\nu}(C_{\nu})$ of $\Sigma$ for all $\nu$.\\
(2) \ For each $p\in \Sigma $ we have $\sum_{\nu} \# \pi _{\nu}^{-1} (p) \le 2$.  
Here $\#$ denotes the order of the set.\\
(3) \ $\sum _\nu \# \pi_{\nu} ^{-1} (z_i) = 1$ for each $z_i$.\\
(4) \ The number of  Riemann surfaces $C_\nu$ is finite.\\
(5) \ The set $\{ p \in \Sigma |\, \sum_{\nu} \# \pi _{\nu} ^{-1} (p) = 2\}$ is finite.

\medskip

We denote by $g_\nu$ the genus of $C_\nu$ and 
by $m_\nu$ the number of points $\overline{p}$ on $C_\nu$, 
which are the inverse image of nodes of $\Sigma$, 
i.e. $\sum_{\gamma} \# \pi_{\gamma}^{-1} (\pi_{\nu}(\overline{p})) = 2$, 
or marked points, 
i.e. $\pi_{\nu}(\overline{p}) = z_j$ for some $j$.  
The genus $g$ of a semistable curve $\Sigma$ is defined by 
$$ g = \sum _\nu g_\nu + \dim H_1 (T_\Sigma, \Q),$$
where $T_\Sigma$ is a graph associated to $\Sigma$ in the following way. 
The vertices of $T_\Sigma$ correspond to the components of $\Sigma$.  
Denote by $v_{\nu}$ the vertex corresponding to $\Sigma_{\nu}$.  
For a node $p \in \Sigma_{\nu} \cap \Sigma_{\nu'}$, 
we assign an edge $e_p$ joining vertices $v_{\nu}$ and $v_{\nu'}$.  
(When $p$ is a node of $\Sigma_{\nu}$, 
the ``edge'' $e_p$ becomes a loop based at $v_{\nu}$.)  

A homeomorphism $\theta : \Sigma \to \Sigma '$ between  two semistable curves is called 
an isomorphism, if it restricts to a biholomorphic 
isomorphism $\theta _{\nu \nu'} : \Sigma _\nu \to \Sigma'_{\nu'}$ for each component 
$\Sigma _\nu$ of $\Sigma$ and some component $\Sigma'_{\nu'}$.  
We also require that $\theta$ maps the marked points in $\Sigma$ onto the corresponding 
marked points in $\Sigma '$ bijectively.

 Let $J(E)$ be a vertical compatible almost
complex structure. A map $u: (\Sigma, z)\to
E$ is called a vertical $J(E)$-stable map, if
the composition map $\pi \circ u$ sends $(\Sigma, z)$ to
a point $b \in B$ and $ u$ is a stable map from 
$(\Sigma, z)$ to $\pi^{-1} (b)= (M,J(E)\vert_{E_b})$.   
In other words, for each $\nu$, the restriction of $u$ to each component 
$\Sigma_\nu$ is either a non-constant map, or we have $m_\nu + 2g_\nu \ge 3$.

To define the moduli space of vertical 
stable maps, we assume first, for the sake of simplicity and a later application, 
that the local system $\Hh_2(E)$ is simple, i.e., the fundamental group $\pi_1(B)$ acts trivially on $\Hh_2(E)$ (e.g. it is the case if the base $B$ is simply connected).

In this case, for a class $A \in H_2(M,\Z)$, there
is  a global locally constant section $s_A$ : $B \to \Hh_2(E)$ whose value is $A$ 
at a reference fiber. 
We consider the moduli space of all vertical stable map $((\Sigma_g, z), u)$ such that $(\Sigma_g, z)$ is of genus $g$ with $m$ marked points.  We denote 
by $C\Mm_{g,m}(E, J(E), s_A)$ the moduli space of vertical
stable maps representing the class $s_A (\pi(u))$:
$$C\Mm_{g,m}(E, J(E),s_A): = \cup_{b\in B} 
\{b\} \times C\Mm_{g,m}(E_b,J(E)\vert_{E_b},s_A(b)),$$ 
which carries a Kuranishi structure in the sense of \cite{F-O1999}, see Lemma 
\ref{Lemma 3.1} below.  

Here  $C\Mm_{g,m} (E_b, J(E)\vert_{E_b}, s_A(b))$  is the moduli  space of stable maps 
of genus $g$, with $m$ marked points and representing the homology class $s_A(b)$.  
Here, two pairs $((\Sigma, z), h))$ and $((\Sigma ', z'), h')$ are equivalent, 
if and only if there exists an isomorphism $\theta: (\Sigma,z) \to (\Sigma', z')$ satisfying $h'\circ \theta = h$.

If  the action of $\pi_1(B)$ on the fiber $H_2(M,\Z)$ is non-trivial, 
we can still define 
a notion of a moduli space of vertical stable maps by considering 
a multi-valued section $s_A: B \to \Hh_2(E)$ 
which is obtained by the locally constant continuation of 
$A$ in a typical fiber to a multi-valued section.  
We note that the pairing $\langle s^{[\om]} (b), s_A (b) \rangle$ as 
well as the pairing 
$\langle s^{c_1} (b), s_A(b) \rangle$ are constant functions on $B$, since they are locally constant 
functions and we assume that $B$ is connected.  
The number $\langle s^{[\om]}, s_A (b) \rangle$ is the ``energy" 
of a holomorphic curve realizing any class in $s_A(b)$, and the second 
number  $\langle s^{c_1}(b), s_A(b) \rangle$ enters 
in the expected dimension of the moduli space of stable maps 
representing any class in $s_A (b)$.  
Now using the Gromov compactness 
theorem it is easy to see that  there is only a finite number of 
values of $s_A$ in each fiber $H_2(M=\pi^{-1}(b), \Z)$ such 
that there is a $J(E)\vert_{E_b}$-holomorphic curve representing 
a homology class in the set $s_A (b)$.  
The projection from $C\Mm_{g,m}(E, J(E), s_A)$ 
to $B$ is proper.  
It follows from the Gromov compactness theorem:  
If $u_i$ is a sequence of $J_i$-stable maps with energy 
bounded by a constant and $J_i$ converges to $J_{\infty}$ in the 
space of almost complex structures, then there exists a subsequence 
$u_{i_k}$, which converges to a $J_{\infty}$-stable map $u_{\infty}$.  

Finally we observe that any element in $s_A (b)$ induces the same class in $H_*(E, \Q)$ by the inclusion.  

\begin{remark}   
It is sometimes more convenient to work with a tame almost complex 
structure (i.e. $\om (X, JX) > 0$ for any non-zero tangent vector $X$).  
As in the non-parameterized case all the compactness 
and perturbation theorems for pseudo-holomorphic curves 
with respect to a compatible almost complex structure hold for a tame 
almost complex structure.  
\end{remark}

\subsection{Examples of symplectic fiber bundles.} 

There are several ways
to construct symplectic fiber bundles.

The first way is the associate bundle method.  
Suppose that  a group $\rm G$ acts
symplectically on  a symplectic manifold $(M,\om)$, i.e. there is a homomorphism
$\rho : {\rm G} \to {\rm Symp} (M,\om)$. Then  we can associate
to each  principal $\rm G$-bundle $P$ over $B$ a symplectic fiber bundle 
$P\times _{\rm G} (M,\om)$.  
This symplectic fiber bundle is non-trivial, if and only if the image $\rho_*(\lambda)$ 
of the homotopy class $\lambda \in [B,{\rm BG}]$ defining the $\rm G$-bundle 
is non-trivial in $[B,{\rm BSymp}(M,\om)]$.

The second way is the pull-back method. 
Suppose that we are given
a symplectic fiber bundle $E$ over a base $B$. Then any map $f$ from
$B'$ to $B$ pulls the bundle $E$ back to a symplectic fiber bundle $E'$ over $B'$.

The third way is the reduction method.  
We begin with a differentiable
fiber bundle $M \to E\to B$ with $M$ being a symplectic manifold and
ask if this fiber bundle also admits a structure of a symplectic fiber bundle.
Of course, it is the case if the inclusion of ${\rm Symp} (M,\om)$ to ${\rm Diff}^+(M)$ is 
a homotopy equivalence
(e.g. if dim $M$ = 2). 
In general, we can state the following criterion, see e.g., 
\cite{G-L-S1996} for more information.  

\
\begin{lemma}\label{Lemma 2.2}
Let $\pi: E\to B$ be a fiber bundle and $\om \in \Om^2(E)$ be a closed 
form such that $\om$ is non-degenerate along all fibers of $E$. 
Then $\pi: E \to B$ admits a structure of a symplectic fiber bundle,  
which is compatible with $\om$. 
\end{lemma}

The fourth way to construct symplectic fiber bundles is the gluing method.
Suppose that we are given two symplectic  bundles $E_1$ and $E_2$ over 
bases $B_1$ and $B_2$ respectively. Suppose that the restriction of
$E_1$ over the boundary $\p B_1$ is isomorphic to the restriction of
$E_2$ over the boundary $\p B_2$. Then we can glue the bundle $E_1$ with $E_2$
 along the boundary $E_i \vert_{\p B_i}$. In particular when the restriction
of $E_i$ over $\p E_i$ is trivial then the glued bundle is defined uniquely by
a map $\p B_1 \to {\rm Symp}  (M,\om)$. 
If $B_i$ is  closed, we can define the operation of fiber connected sum 
as follows.   
Choose a small disk $D_i$ in $B_i$ and take a trivialization of $E_i\vert_{D_i} \to D_i$.  
Then glue $E_i\vert_{B_i \setminus D_i}$, $i=1,2$, along $\p D_i$.

\begin{remark}  
Each element $g\in \pi_k ({\rm Symp} (M,\om))$ defines
a symplectic fiber bundle $E$ with the fiber $(M,\om)$ over a sphere
$S^{k+1}$ by gluing two trivial symplectic fiber bundles $D^{k+1} \times M$
along the boundary $M\times S^k$ by the element $g$. 
Conversely any symplectic fiber bundle over $S^{k+1}$ is defined by such a method.
\end{remark}

\section{Parameterized Gromov-Witten invariants}

In this section we define parameterized Gromov-Witten invariants 
for symplectic fiber bundles over a closed oriented manifold $B$.  
The base $B$ is assumed to be oriented in oder to deal with 
the orientation of the moduli space of stable maps. 
The base $B$ is also assumed to be a closed manifold in order to get 
the fundamental class of the moduli space of vertical stable maps.

\subsection{Geometric picture}

Recall that a semistable curve $(\Sigma, z)$ with $m$ marked points is called stable, 
if for all its component $C_\nu$ of the normalization of $\Sigma$ we have
$m_\nu + 2g_\nu \ge 3$.
Let $C\Mm_{g,m}$ denote the Deligne-Mumford moduli
space of stable curves, i.e. $C\Mm_{g,m}$ is the set of all isomorphism classes of stable curves with $m$ marked points and of genus $g$. Let us denote by $E^{(m)}$ 
the ``Whitney sum'' (the multiple fiber product over $B$) of $m$ copies of $E$.  
When $2g+m \geq 3$, 
as in the usual case (see e.g. \cite{K-M1994}, 2.4, \cite{Kontsevich1995}, 1.5), 
there is the evaluation map
$$\Pi= pr \times ev_{g,m,s_A} : C\Mm_{g,m}(E, J(E), s_A) \to C\Mm_{g,m} \times E^{(m)},$$
$$ ((\Sigma,z), u)\mapsto ((\tilde \Sigma, \tilde z), u(z_1), \cdots , u(z_m)).$$
Here $(\tilde \Sigma , \tilde z)$ is the stable curve with marked points obtained from 
$(\Sigma, z)$ by  consecutive contractions
of  non-stable components.   
When $2g+m = 0, 1$, we call $ev_{g,m,s_A}$ the evaluation map.  

We briefly recall the notion of Kuranishi structures, 
see \cite{F-O1999}, \S 5 for details.  
Roughly speaking a compact Hausdorff space $X$ 
has a Kuranishi structure, if it is locally described as the zero set 
$s ^{-1} (0)$  of a $V$-bundle over a $V$-manifold, 
namely for each $p \in X$, there exist a $V$-manifold $U_p$, 
a $V$-bundle $E_p$ on it and a continuous section $s$ of $E_p \to U_p$ 
such that the difference $\dim E_p -\dim U_p$ is independent of 
$p \in X$.  
Moreover, we assume that 
such local descriptions are compatible under 
coordinate changes $\{\phi_{pq}\}$ in a suitable sense.  
We also have the notion that $X$ with Kuranishi structure 
has its tangent bundle.  
If a continuous map $f:X \to Y$ extends locally to 
$f_p : U_p \to Y$ for each $p$ 
such that $f_p \circ \phi_{pq} = f_q$, 
we call $f$ is a strongly continuous map.  
If $X$ has an oriented Kuranishi structure and $f$ is 
a strongly smooth map from $X$ to a topological space $Y$, 
then we can define the image  $f_* ([X])$ of the fundamental class of $X$ 
as the image $f_*[(s')^{-1} (0)]$ of the fundamental class of the zero set $(s')^{-1} (0)$ of a perturbed smooth multi-section $s'$ of $E$  
which is transversal to zero. 
Taking into account of multiplicity 
in an appropriate way, this
fundamental class gives a well-defined element in  $H_*(Y, \Q)$.

\begin{lemma}\label{Lemma 3.1}
The space $C\Mm_{g,m} (E, J(E), s_A)$ has a Kuranishi structure with oriented 
tangent bundle.  
This space is  compact and of dimension $\dim B + 2m + 
2 \langle c_1 (M), s_A \rangle + (6-\dim M) (g-1)$ .  
Moreover, $\Pi$ is strongly continuous map in the sense of Kuranishi structure.  
\end{lemma}

By this Lemma, we can define the virtual fundamental 
cycle of the moduli space of vertical stable maps.  
Denote by $\Pi_*([C\Mm_{g,m}(E,J(E),s_A)])$ the induced 
class in $H_*(C\Mm _{g,m} \times E^{(m)}, \Q)$. 
The map $\Pi$ induces a map in cohomologies

\begin{equation} 
I^E_{g,m,s_A}: H^*(E^{(m)},\Q)  \to H^{*+\mu}(C\Mm_{g,m},\Q) 
\label{3.1.2}
\end{equation} 
by
\begin{equation}
I^E_{g,m, s_A}(\gamma)=PD(\gamma \backslash \Pi_*(C\Mm_{g,m}(E,J(E),s_A))) 
\label{(3.1.3)}
\end{equation}
If $2g+m = 0, 1$, we define 
\begin{equation}
I^E_{g,m,s_A}(\gamma) = \langle \gamma,
(ev_{g,m,s_A})_*(C\Mm_{g,m}(E,J(E),s_A)) \rangle \in \Q. \label{(3.1.3')} 
\end{equation}
The parameterized Gromov-Witten
invariants, as in the usual case, are the collection
of maps $I^E_{g,m,s_A}$ defined in (\ref{(3.1.3)}), (\ref{(3.1.3')}). 
The shift of grading 
$\mu$ equals $-\langle 2c_1(M), A \rangle + (g-1)\dim M -\dim B$.

\begin{remark} 
For symplectic fiber bundles $E$  with 
the local system $\Hh_2(E)$ being simple, we can also interpret 
the parameterized Gromov-Witten invariants $I^E_{g,0, s_A}$ 
with $\mu=0$ (relative degree $0$) in term of 
``counting vertical holomorphic curves" 
of genus $g$ representing any class in $s_A$.  
Here ``counting'' means the ``order'' of the space with Kuranishi 
structure of expected dimension $0$.  
\end{remark}

\subsection{Proof of Lemma \ref{Lemma 3.1}}  
To define rigorously the parametrized Gromov-Witten invariants 
for all compact symplectic fiber bundles, we need to prove Lemma 
\ref{Lemma 3.1} and moreover, to show that
the map $I^E_{g,m,s_A}$ does not depend on the choice of $J(E)$. 
Since the base space $B$ is assumed to be compact, the moduli space 
$C\Mm_{g,m}(E, J(E), s_A)$ is compact, cf. section 2.2.  

\begin{proof}[Proof of Lemma \ref{Lemma 3.1}] 
Our proof of Lemma \ref{Lemma 3.1} is  an adaptation of the proof of 
the corresponding results concerning Gromov-Witten invariants 
\cite{F-O1999}, Theorems 7.10 and 7.11.  

Let $u$ be a vertical stable map over $b_0 \in B$.  
For simplicity, we assume that the domain of $u$ is irreducible.  
(The general case is handled in a similar way.)
Pick a finite dimensional space 
$E_0 \subset L^p\Om^1(u^*TE_{b_0})$ such that 
$$
{\rm Im}~D_u\overline{\partial}_{J_{b_0}} + E_0 = L^p\Om^1(u^*TE_{b_0}).
$$
If $b \in B$ is in a neighborhood $D_0$ of $b_0$, we can identify fibers 
$E_b$ and $E_{b_0}$ by a local trivialization of $E$.  
Thus $J_b$ is considered as an almost complex structure on $E_{b_0}$.  
Then there is a smaller neighborhood $D'_0 \subset D_0$ of $b_0$ 
such that 
$$
{\rm Im}~D_u\overline{\partial}_{J_{b}} + E_0 = L^p\Om^1(u^*TE_{b_0}).
$$
Using this observation, the argument in \cite{F-O1999} implies 
that there exists 
a Kuranishi neighborhood $(U, \Ee, s)$ of $u$ on 
$C\Mm_{g,m}(E, J(E), s_A)$ 
so that a $V$-manifold $U$ is fibered over an open subset of $B$.   

All arguments in \cite{F-O1999} can be directly adapted to 
the parametrized case or they even imply the corresponding statements in 
the parametrized case.  
For details of the construction of Kuranishi structure, see \cite{F-O1999}, \S 12.   
\end{proof}

In order to show that the usual Gromov-Witten invariants 
do not depend on the choice of compatible almost complex structures, 
perturbations, 
we needed to construct a bordism between the moduli spaces 
corresponding to two almost complex structures $J_1$ and $J_2$ 
(with perturbations).   
This bordims is a version of the moduli space 
of vertical stable pseudo-holomorphic 
curves parametrized by the interval $[0,1]$.

\subsection{Properties of parameterized Gromov-Witten invariants}

The following theorem immediately follows from Lemma \ref{Lemma 3.1} 
applied to the case that the base space is $B \times [0,1]$.  

\begin{theorem}\label{Theorem 3.3}
The parameterized Gromov-Witten invariants are
 invariants of symplectic fiber bundles. 
\end{theorem}

\begin{remark}
The bordism type invariants
of the moduli space of pseudoholomorphic curves
may have more informations than the (cohomological) Gromov-Witten 
invariants. 
Such  examples of finer Gromov-invariants of 
symplectic manifolds  can be found in \cite{McDuff1987}.
\end{remark}

In order to distinguish a symplectic fiber bundle from 
the trivial one by parameterized Gromov-Witten invariants, 
we need to compute those for trivial symplectic fiber bundles 
$E = B \times M$. 
By the K\"unneth formula, the algebra $H^*(B \times (M)^{(m)}, \Q)$ 
is isomorphic to $H^*(B,\Q) \otimes (H^*(M, \Q))^{\otimes m}$. 
Let us denote by $\alpha_i$ elements in $  H^*(M,\Q)$ and by $\beta$ an element 
in $H^*(B,\Q)$.

\begin{proposition}
The Gromov-Witten invariants of a trivial symplectic fiber bundle
equals
\begin{equation}
I^E _{g,m, s_A} (\beta \otimes \alpha_ 1 \otimes \cdots \otimes \alpha_m) 
= 
I^M_{g,m, A} (\alpha_1 \otimes \cdots \otimes \alpha_m)\int_B \beta . 
\label{(3.3.4)}
\end{equation}
\end{proposition}

\begin{proof} 
We choose  a  vertical compatible almost complex structure
$J(E=B \times M)$ such that it is constant in the $B$-direction.  
Clearly the moduli space
of vertical stable maps $C\Mm _{g,m}(E=B \times M,\{J_b \equiv J\}, s_A)$ is
the direct product $B \times C\Mm_{g,m}(M,J,A)$. 
We take a multi-valued perturbation of Kuranishi map for 
$C\Mm_{g,m}(M,J,A)$ to define the virtual fundamental cycle of 
$B \times C\Mm_{g,m}(M,J,A)$.  
Denote by $\Pi^{pt}$
the evaluation map in section 3.1 for $C\Mm_{g,m}(M,J,A)$, i.e. 
the case with base $B= pt$. By (\ref{(3.1.3)}), 
the left hand side  of
(\ref{(3.3.4)}) equals
\begin{equation}
\begin{split}
& PD(\beta \otimes \alpha_1\otimes \cdots \otimes \alpha _m \backslash \Pi_*
[B \times C\Mm_{g,m}(M,J,A)]) \\
= & PD(\beta \otimes \alpha_1\otimes \cdots \otimes \alpha_m \backslash
([B]\times \Pi^{pt}_* [C\Mm_{g,m}(M,J,A)]) \label{(3.3.5)}
\end{split}
\end{equation}

Clearly the right hand side of (\ref{(3.3.4)}) equals the right hand 
side of (\ref{(3.3.5)}). 
\end{proof}

Now let us compute parameterized Gromov-Witten invariants of
a pull-back symplectic fiber bundle. 
Let $p:B_1 \to B_2$ be a $k$-fold covering space and $E_2 \to B_2$ a symplectic 
fiber bundle.  
Then the pull-back $E_1=p^*E_2 \to B_1$ is also a symplectic fiber bundle.  
For a single section $s_A$ of $\Hh_2(E_2)$, denote by $p^*s_A$ its pull back.  
We get immediately the following 

\begin{proposition} 
We have
$$I^{E_1}_{g,m,p^*(s_A)} = k I^{E_2}_{g,m,s_A}.$$
\end{proposition}

Parameterized Gromov-Witten invariants of relative degree 0 and without marked points satisfy the following additivity.

\begin{proposition}
Let $E = E_1\# E_2$ be  a fiber
connected sum of symplectic fiber bundles $E_1$ and $E_2$. Then we
have the following formula for parameterized Gromov-Witten invariants 
of relative degree 0.
$$I^E_{g,0,s_A} = I^{E_1}_{g,0,s_A} + I^{E_2}_{g,0,s_A}.$$
\end{proposition}

\begin{proof}   
By the dimension assumption of the Gromov-Witten invariants, we take perturbation, 
if necessary, so that there is no vertical stable curves representing the class $A$ 
over a small disk $D_i(\eps)$ in the base $B_i$ for $i =1,2$.   
We can assume further that our fiberwise almost complex structures on 
$D_i(\eps)$, $i=1,2$, are constant and isomorphic each other.   
Now we perform the connected sum of symplectic fiber bundles  using these disks. The almost complex structures  on $E_i$  can be glued together  identifying
their restrictions on $D_i (\eps) \times M$.  Hence
we obtain the proposition. 
\end{proof}

\begin{remark}For symplectic fiber bundles $\pi_i: E_i \to B_i$ with simple 
local systems $\Hh^*(E_i)$, $\alpha_j \in H^{q_j}(M;\Q)$ defines 
the locally constant sections of $\Hh^{q_j}(E_i)$, 
$i=1,2$, and $\Hh^{q_j}(E_1 \# E_2)$.   
Suppose that there exist cohomology classes $\widetilde{\alpha}_j^{(i)} 
\in H^{q_j}(E_i;\Q)$ such that their restrictions to typical fibers 
coincide with $\alpha_j$.  
Let $\widetilde{\alpha}_j \in H^{q_j}(E_1 \# E_2;\Q )$ be 
a cohomology class, which is equal to $\widetilde{\alpha}_j^{(i)}$ 
after restricting to $\pi_i^{-1}(B_i \setminus D_i(\eps ))$.  
When $\mu + \sum_j q_j = 6 (g-1)$, 
we have 
$$
I^E_{g,m,s_A} (\prod_j \widetilde{\alpha}_j) = I^{E_1}_{g,m,s_A} 
(\prod_j \widetilde{\alpha}_j^{(1)}) + I^{E_2}_{g,m,s_A} (\prod_j \widetilde{\alpha}_j^{(2)}).  
$$  
\end{remark}

At the end of this section we would like to suggest that many properties of the Gromov-Witten
invariants (e.g. the Kontsevich-Manin axioms) should be valid in the family version.
Specially interesting seems to us an analog of Taubes' theorem on the relation of Gromov-Witten
invariants and Seiberg-Witten invariants in dimension 4.   
It would imply that the parametrized Gromov-Witten
invariants also bring information on the homotopy type of 
the diffeomorphism group of 4-dimensional symplectic manifolds.   

\section{Homotopy groups of symplectomorphism groups.}

In this section we combine Remark 2.3, Propositions 3.5, 3.6, 3.7
and other observations to 
estimate the rank of homotopy groups of symplectomorphism groups. 

\begin{proposition}\label{Prop 4.1}
Parameterized Gromov-Witten invariants of relative degree $0$ and 
without marked points $I^E_{g,0,s_A}$ over sphere $S^{i+1}$
define elements in $Hom (\pi_i({\rm Symp} (M,\om)),\Q)$, $i\geq 1$.
\end{proposition}

\begin{proof}
Denote by $E_g$ the symplectic fiber bundle over $S^{i+1}$ by gluing 
$D^{i+1}_1\times M \cup_g D^{i+1}_2\times M$ with the help of a map $g: \p D^ {i+1}_1 \to {\rm Symp}  (M)$,
i.e. we identify a pair $(x,y)\in  \p D^{i+1}_1 \times M$ with $(x, g(x) \cdot y) \in
\p D^{i+1}_2\times M$.  
Then we have
$E_{g\cdot f} \cong E_g \# E_f$.   
Now we get Proposition 4.1 immediately from Remark 2.3
and Proposition 3.7.  
\end{proof}

\begin{remark}\label{Rem 4.1.1}
Taking into account Remark 3.8 we can get a similar statement 
for parameterized Gromov-Witten invariants of relative degree 0.  
We use such invariants in the proof of Theorem \ref{Theorem 4.4}.a.
\end{remark}

As an application of Proposition \ref{Prop 4.1}, Remark \ref{Rem 4.1.1}, 
we shall prove Theorem \ref{Theorem 4.3} and Theorem \ref{Theorem 4.4}.

Let $(M,\om) = (S^2\times S^2, \om^{(1)} \oplus \om^{(2)})$ be a product 
of symplectic manifolds. 
We denote by $A_i$ the generators of $H_2(M,\Z)$ realizing by 
the $i$-th sphere $S^2$. 

\begin{theorem}\label{Theorem 4.3}
If $\om^{(1)} (A_1) > \om^{(2)} (A_2)$ then there is an 
$S^2\times S^2$-symplectic fiber bundle over $S^2$ with non-vanishing   
parameterized Gromov-Witten invariants of relative degree $0$ 
and $m=0$.
In particular the rank of $\pi_1({\rm Symp}  (S^2\times S^2, \om))$ is
at least 1.
\end{theorem}

We recall that the last statement in Theorem \ref{Theorem 4.3} is 
established in Theorem 2.4. C$_2$ in \cite{Gromov1985}.

{\it Proof of Theorem \ref{Theorem 4.3}}.  
There are several ways of describing  the proof of
this theorem (see also Appendix, which 
follows the original idea of Gromov \cite{Gromov1985}).  
We present here a proof following 
the idea of McDuff in \cite{McDuff1987}, Lemma 3.1, which  
uses the deformation space of 
complex structures of Hirzebruch's surfaces of even degree, 
which is diffeomorphic to $S^2\times S^2$, see \cite{M-K1971}.  
Thus we can apply technique in complex analytic geometry for our computation.

Denote by $\Oo(\ell)$ the holomorphic line bundle of degree $\ell$ 
on $\C P^1$.  
Let us recall that the Hirzebruch surface $F_k$ is
the projectivization of a rank 2  holomorphic vector bundle 
$W_k=\Oo(0)\oplus \Oo (k)$ with $k \geq 0$ over $\C P^1$. 
The line subbundles $\Oo(0) \oplus 0$ and $0 \oplus \Oo(k)$ define 
sections of $F_k={\bf P}(W_k) \to \C P^1$ 
with self-intersection number $k$ and $-k$, respectively.  
We denote by $J^{F_k}$ the complex structure on the Hirzebruch surface 
$F_k$.  
It is known that all Hirzebruch surfaces with 
even degree $k$ are diffeomorphic to $S^2\times S^2$, and $F_0$ 
is biholomorphic 
to ${\C}P^1 \times {\C}P^1$ (see e.g., \cite{M-K1971}, chapter 1).

We define a family of holomorphic vector bundles $\{V_a\}$ of rank 2 
on $\C P^1$ as follows.  
Write $U_0 = \C P^1 \setminus \{\infty \}$ and 
$U_{\infty}=\C P^1 \setminus \{ 0\}$.  
Consider the transition function $f_a: (U_0 \setminus \{ 0\}) \times 
\C^2 \to (U_{\infty} \setminus \{\infty \}) \times \C^2$ by 
$$
f_a(z,v_1,v_2)=(z,zv_1 + av_2, z^{-1}v_2).
$$
Denote by $X_a$ the projectivization of $V_a$ and $J_a$ the complex 
structure on it.  
Note that $V_{a=0}$ is isomorphic to $\Oo(-1) \oplus \Oo(1)$ 
and that ${\bf P}(V_{a=0})$ is isomorphic to ${\bf P}(W_2)$.   
Thus $\{X_a\}$ is a complex one-dimensional deformation of 
$X_0=F_2$.  
The complex structure $J_0$ is the complex structure $J^{F_2}$.  
All $J_a$, $a \neq 0$, are isomorphic to the complex structure $J^{F_0}$, 
i.e., the product complex structure.  

Note that 
$$\Xx = \cup_{a\in {\bf C}^1} \{a\} \times X_a \to {\bf C}^1 \times {\bf C}P^1$$
is the projectivization of the vector bundle 
$$\Vv = \cup_{a\in {\bf C}^1} \{a\} \times V_a \to {\bf C}^1 \times {\bf C}P^1.$$  
Since the restriction of the vector bundle $\Vv$ to 
$({\bf C} \setminus \{0\}) \times {\bf C}P^1$ is 
holomorphically trivialized 
by the following 2 sections $\sigma_1$ and
$\sigma_2$, which are everywhere linearly independent:

\begin{equation}
\sigma_1 (a,z) = 
\begin{cases}
(z,0,1) \in U_0 \times \C^2, \ \ {\rm if} \ z \in U_0 \\
(z,a,\frac{1}{z}) \in U_{\infty} \times \C^2, \ \ {\rm if} \ z \in U_{\infty}
\end{cases} 
\nonumber
\end{equation}

\begin{equation}
\sigma_2 (a,z) = 
\begin{cases}
(z,1,-\frac{z}{a}) \in U_0 \times \C^2, \ \ {\rm if} \ z 
\in U_0 \\
(z,0,-\frac{1}{a}) \in U_{\infty} \times \C^2, \ \ {\rm if} \ z \in 
U_{\infty} 
\end{cases} 
\nonumber
\end{equation}

Note that 
$$f_a(z,0,1) = (z,a,\frac{1}{z}), \ \ 
f_a(z,1,-\frac{z}{a})=(z,0,-\frac{1}{a}).$$
Thus $\sigma_1, \sigma_2$ are well-defined holomorphic sections.  
Using this trivialization, we extend $\Vv$ and $\Xx$ across 
$\{\infty\} \times \C P^1$ to $\C P^1 \times \C P^1$.  
We denote these extensions by the same symbols.  
$\Xx$ is also considered as an $S^2 \times S^2$-bundle with the projection 
to the first factor $\C P^1$ parameterized by $a$.  
Denote this fiber bundle by $E \to \C P^1$.  
The complex structure on $\Xx$ induces the fiberwise complex structure $J(E)$.  

Clearly the parameterized Gromov-Witten invariant $I_{0,0,A_1-A_2}$  
is of relative degree 0.  
We shall show that its value $I^E_{0,0,A_1-A_2}$ is 1. 
Note that there is no $J^{F_0}$-holomorphic sphere realizing  class $(A_1-A_2)$, 
since $J^{F_0} = \C P^1 \times \C P^1$.  
Further, there is exactly one $J^{F_2}$-holomorphic sphere
realizing class $(A_1-A_2)$, which is the  section of self-intersection number $-2$ 
in the $\C P^1 $ bundle over $\C P^1$. 
Thus the moduli space
$C\Mm_{0,0, A_1-A_2}(E, J(E))$ consists of one point at $a=0$.  

To prove the transversality of this moduli space, we argue as follows.
Let $N$ be the normal bundle of the unique $J^{F_2}$-holomorphic sphere 
in the class $(A_1-A_2)$ in the total space of fibration.  
In order to show the transversality in the case of an integrable
complex structure, it suffices to prove the following 

\begin{lemma}
We have $H^1(\C P^1;\Oo(N))=0$.  Hence, 
$N = \Oo(-1) \oplus \Oo (-1)$. 
\end{lemma}

\begin{proof}
Consider the cohomology exact sequence associated to
$$0 \to \Oo(-2) \to N  \to \Oo _D \to 0 ,$$
where $\Oo (-2)$ is the normal bundle of the $(-2)$-curve in the central 
fiber $X_{a=0} \subset \Xx$ with the complex structure $J^{F_2}$,  
and the third term $\Oo_D$ is the quotient,
which is nothing but the pull back of the normal bundle of the
origin in $\C P^1$ parameterized by $a$.  
We shall show that  the connecting homomorphism
$H^0 (\Oo_D) \to H^1 (\Oo(-2))$ is surjective, (hence isomorphism).  
As a consequence we get $H^1 (N)=0$.

Here we regard $X_a$ as the projectivization of 
$\Oo(-1) \otimes V_a$.  When $a=0$, it is isomorphic to 
$\Oo(-2) \oplus \Oo(0)$.  
The vector bundle $\Oo(-1) \otimes V_a$ is written as 
the gluing $U_0 \times \C^2$ and $U_{\infty} \times \C^2$ by 
$$(z,v_1,v_2) \mapsto (z, z^2 v_1 + a z v_2, v_2).$$ 
Note that the $(-2)$-curve representing $A_1 - A_2$ is the image of 
$\{0\} \otimes \Oo(0)$, 
hence the image of the section $(0,1)$ in 
${\bf P}(\Oo(-2) \oplus \Oo(0))$ over $a=0$.  

Consider the restrictions of the section $(0,1)$ of 
$\Oo(-1) \otimes V_{a=0}$ over $U_0$ and $U_{\infty}$, respectively.  
For $a \in \C$, the pair $z \in U_0 \mapsto (z,0,1) \in U_0 \times \C^2$ and 
$z \in U_{\infty} \mapsto (z,0,1) \in U_{\infty} \times \C^2$ gives a deformation 
as a Cech $0$-cocycle.  
Taking the Cech differential, we get a Cech $1$-cocycle 
$z \in U_0 \cap U_{\infty} 
\subset U_{\infty} \mapsto (z, az, 0) \in (U_{\infty} \setminus 
\{\infty\}) \times \C^2$.  
Differentiate in $a$, then we find that the last component 
of the Cech $1$-cocycle vanishes and obtain a Cech $1$-cocycle 
$z \in U_0 \cap U_{\infty} 
\subset U_{\infty} \mapsto (z, z) \in (U_{\infty} \setminus 
\{\infty\}) \times \C$ of $\Oo(-2)$, which represents a non-zero element in 
$H^1(\C P^1;\Oo(-2))$.  
Hence we conclude that the connecting homomorphism 
$H^0 (\Oo_D) \to H^1 (\Oo(-2))$ is surjective.  
\end{proof}

Let us generalize Theorem \ref{Theorem 4.3}. We denote by $(X_1^4,\om_1)$ a non-monotone symplectic manifold
which is diffeomorphic to $S^2\times S^2$ and by $(X_2^4,\om_2)$ a symplectic manifold whch is diffeomorphic
to $\C P^2\# \overline{\C P^2}$. We are going to prove the following 

\begin{theorem}\label{Theorem 4.4}
a) Let $(M_1,\Om_1)$ $= (X_1^4\times N^{2k},\om_1\oplus \om_0)$ be 
a symplectic manifold with $(X_1,\om_1)$ as above and 
$(N^{2k},\om)$  a compact symplectic manifold. 
Then we have $rk (\pi_1 ({\rm Symp}  (M_1,\Om_1))\otimes \Q) \ge 1$.

\noindent
b) We also have $rk ({\rm Symp} (M^4_2, \om_2))\ge 1.$
\end{theorem}

\begin{proof}
a) To prove the  statement for $(M_1,\Om_1)$ we shall construct a symplectic fiber 
bundle $E$ with fiber $(M_1,\Om_1)$
over $S^2$ and compute  a parameterized Gromov-Witten
invariant of relative degree 0 and with one marked point of $E$.  
For  a symplectic fiber bundle $E$ with fiber $(M_1,\Om_1)$,  
we consider  the moduli space of vertical
holomorphic mappings $f: S^2 \to E$ whose 
image represents the class $A:=A_1 -A_2$. 
Note that the local system $\Hh_2(E)$ is 
simple for any symplectic fiber bundle 
over the base space $S^2$.  

The dimension computation shows us that the moduli space of
vertical stable maps of genus $0$ in the class $A$ on 
an $(M,\om)$-bundle over $S^2$ has dimension $2k= \dim N^{2k}$.  
We will compare symplectic fiber bundles over $S^2$, 
each of which is the product of a symplectic fiber bundle $E'$ with 
the fiber $X_1^4$ over $S^2$ and $N=N^{2k}$, i.e., 
$E=E' \times N \to S^2$.  
We shall count the number of vertical  
pseudo-holomorphic spheres $u$ with one marked point $z$ 
in the class $A$ so that $ev(u;z) \in \Gamma$.   
Here $\Gamma$ is a cycle represented by 
a submanifold $E' \times \{ y_0 \} \subset E$, for some 
arbitrary chosen point $ y_0 \in N^{2k}$. 
This number is the parameterized 
Gromov-Witten invariant $I^E _{0,1, A_1-A_2} (PD [\Gamma])$.  

Note that the restriction of $PD [\Gamma] \in H^*(E)$ 
to $E|_b$, $b \in B$, is $PD [X_1^4 \times \{y_0 \}] 
\in H^*(X_1^4 \times N)$.  But this condition does not 
characterize $PD [\Gamma] \in H^*(E)$.  Let $\Gamma' \in E$ be 
another cycle such that the restriction of $PD [\Gamma']$ to 
$E|_b$ is equal to $PD [X_1^4 \times \{y_0 \}] \in H^*(X_1^4 \times N)$. 
Since the base space $B$ is $S^2$, we find that 
$PD [\Gamma'] - PD [\Gamma] \in H^2(B) \otimes H^*(X_1^4 \times N)$.  
In other words, $[\Gamma'] -[\Gamma]$ is represented by some cycle 
contained in a single fiber $E|_{b_0}$.  
By dimensional counting argument, we can take a fiberwise almost complex 
structure such that there are no pseudo-holomorphic $(A_1 - A_2)$-spheres 
contained in $E|_{b_0}$. Thus we have 
$I^E _{0,1, A_1-A_2} (PD [\Gamma])=I^E _{0,1, A_1-A_2} (PD [\Gamma'])$.  
Hence $I^E _{0,1, A_1-A_2} (PD [\Gamma])$ gives an invariant 
for symplectic $X_1^4 \times N$-bundles over $S^2$.  

We claim that this number of the 
the trivial bundle $S^2\times (M_1,\Om_1)$ equals zero.  
To show it we consider 
a vertical almost complex structure  $J^{\rm prod}$ on the trivial bundle 
$(M,\om)\times S^2$ such that 
on each fiber $(M,\om)$ we have $J^{\rm prod}= J^0 \times J^N$, where 
$J^0$ is the standard product complex structure on $S^2\times S^2$ and 
$J^N$ is an almost complex structure on $(N^{2k},\om_0)$.  
Clearly the projection on the first factor of any $J^{\rm prod}$-sphere 
is also a $J^0$-sphere in $S^2\times S^2$.  
Hence the moduli space of $J^{\rm prod}$-spheres in class $A_1-A_2$ 
is empty. 

Now we  construct a symplectic $(M_1,\Om_1)$-bundle  $E$ over
$S^2$ by gluing two trivial symplectic $(M_1,\Om_1)$-bundles, 
one over a disk $B^2$ 
and the other over a disk $D^2$, using the loop $\tilde g_t :S^1\to 
{\rm Symp}(M_1,\Om_1)$ of the form $\tilde g_t = (g_t\times \{ Id\}) \in ({\rm Symp} (X_1^4 , \om_1)\times \{ Id\})
\subset {\rm Symp} (M_1, \Om_1)$.  
Here $g_t$ is the transition function for $\Xx \to \C P^1$.  
(See also Appendix.)  

To compute the Gromov-Witten
invariant $I^E_{0,1,A_1-A_2}(PD[\Gamma])$,
we  construct a fiberwise 
compatible almost complex structure $J(E)$ by gluing two 
fiberwise compatible almost complex structures on 
the restriction of $E$ to  disks $D^2_0$ and $D^2_1$.  
The first fiberwise compatible almost complex structure is defined 
as follows: $J(z)= (J^0  \times J^N)$ for $z\in D^2_0$. 
The second one is defined as follows: $J(a)=(J_a \times  J^N)$, 
$a \in D^2_1 \subset \C$.  
They are glued by using the symplectomorphism loop $\tilde g_t$.  

We claim that the constructed vertical 
almost complex is $A=A_1-A_2$-generic.  
Clearly outside the  singular point  $z =0$ 
in the disk $D^2_1$ of the base $S^2$, 
where the vertical almost complex structure take 
value $(J^{F_2} \times J^N)$, the moduli space of $J(E)$-holomorphic 
spheres realizing $A$ is empty. At the singular point $a = 0$ 
the moduli space is diffeomorphic to $N^{2k}$, namely it consists 
of maps $u_y(t)= \{ u_1(t)\times y\}, y\in N^{2k}$, where $t\in S^2$ 
and $u_1$ is the $J^{F_2}$-holomorphic $(-2)$-sphere in $X^4_1$.  
Clearly the transversality of the constructed $J(E)$ is equivalent 
to the surjectivity of the linearization map 
$D_u\bar{\p}_{J(E)}: T_{\pi (u)} S^2 \times L^p_1 (u^* T_{ver}E) 
\to L^p (\Lambda ^{0,1} S^2\otimes _{J(E)} u^* T_{ver}E)$.  
Since  $u= (u_1,y)$, 
we have $u^* T_{ver}E = (u_1 ^*TX^4_1 )\times T_y N^{2k}$, 
so the surjectivity of $D\bar {\p}_{J(E)}$
follows from the surjectivity of the linearization map considered in 
the proof of Theorem \ref{Theorem 4.3}, see also the proof of Lemma 5.1 in Appendix.  
This proves the first statement in Theorem \ref{Theorem 4.4} for the case 
$(M_1,\Om_1)$.

Now let us prove the statement b).
Denote by ${\rm Symp}  (\C P^2, \om, pt)$ the subgroup of the symplectomorphisms
of $(\C P^2,\om)$ which preserve a point $pt$.
Clearly Theorem \ref{Theorem 4.4}.b follows from the following Lemmas 4.5, 4.6.
\end{proof} 

\begin{lemma}\label{Lem 4.6}
The  fundamental group of ${\rm Symp}(\C P^2, \om, pt)$ contains
a subgroup $\Z$. 
\end{lemma}

\begin{lemma}
There is an injective homomorphism $\alpha$ from the infinite 
cyclic subgroup of $\pi_1({\rm Symp} (\C P^2, \om, pt))$ in Lemma \ref{Lem 4.6} 
to $\pi_1({\rm Symp}  (X^4_2,\om_2))$. 
\end{lemma}

Lemma \ref{Lem 4.6} is a direct consequence of Gromov's theorem, 
which states that ${\rm Symp} (\C P^2, \om)$ is homotopy equivalent to 
$PU(3)$, since the quotient space 
${\rm Symp} (\C P^2, \om)/{\rm Symp} (\C P^2, \om, pt)$ is 
isomorphic to $\C P^2$.  We can also see this fact as follows.  
It suffices to show that the inclusion $U(2)\to {\rm Symp} (\C P^2, \om, pt)$
induces an injective homomorphism on the corresponding fundamental groups.
To see it we consider the evaluation map $ev: {\rm Symp} (\C P^2, \om, pt) \to Sp(4): \: g \mapsto
Dg(pt, v)$, where $v$ is an element in the frame $Sp(4)$ over the fixed
point in $\C P^2$. The restriction of this evaluation map to $U(2)$ is injective, and we know that the image $ev(U(2))$ is a deformation retract of $Sp(4)$. Hence follows the Lemma.

\begin{proof}[Proof of Lemma 4.7] 
Denote by $E_a$ the symplectic fiber bundle
over $S^2$ with fiber $(\C P^2, \om ,pt)$ corresponding to element
$a\in \Z = \pi_1 (U(2)) \subset \pi_1({\rm Symp} (\C P^2, \om, pt))$.  
Using the spectral sequence for $E_a$ 
we see that there is a closed $2$-form  $\Om$ such that 
 the restriction of $\Om$ to the fiber $\C P^2$ equals $\om$.   
Thus we can apply the Thurston construction (see e.g. \cite{McD-S1995}, p. 193) 
to conclude that 
$E_a$ is a symplectic manifold with a symplectic form 
$\Om_K$ in a class $K\pi^*(\om_0) + \Om$, where $\om_0$ is a symplectic form on the base $S^2$.  
Moreover, all the fibers $\C P^2 $ are symplectic submanifolds of $(E_a, \Om_K)$.  
Let $s$ be a section $S^2\to E_a$  such that $s(S^2)$ is a symplectic 
submanifold of $(E_a,\Om_K)$.   
Next we construct a new symplectic fiber bundle over $S^2$ by fiber-wise 
blowing-up a symplectic fiber bundle with the fiber $(\C P^2, \om)$ at point 
$s(x), x\in S^2$.  
That is exactly the blow-up $(E_a, \Om_K)$ along the submanifold  $s(S^2)$.  
Denote by $\widetilde{E}_a, \tilde{\Om_K}$ the result 
of blowing-up.  
Clearly the fiber of this new fiber bundle is $\C P^2\# \overline{\C P^2}$.  
Now we apply Lemma \ref{Lemma 2.2} to conclude that $\widetilde{E}_a$ is a symplectic 
fiber bundle with the fiber $(X^4_2, \om_2)$. Thus we have constructed a map 
$\alpha : \Z \subset \pi_1({\rm Symp} (\C P^2, \om)) \to \pi_1({\rm Symp} (X^4_2, \om_2))$.
Because our
construction is compatible with the operation of fiber-connected sum we
conclude that $\alpha$ is a homomorphism. 

Now assume that $\alpha $ is not injective.
Then, for some integer $a \neq 0$, 
we can find a trivialization of symplectic fiber bundle $\widetilde{E}_a$ with 
a constant compatible complex structure $J_{\widetilde{E}_a}^0$.   
We denote 
by $\{ J_{\widetilde{E}_a}^t \}_{0 \leq t \leq 1}$ 
a family of compatible almost complex structures
on $\widetilde{E}_a$ with $J_{\widetilde{E}_a}^1$ 
being the almost complex structure 
resulting from the blow-up process along $s(S^2)$.  
Note that the space of compatible almost complex structures, 
which admit $(-k)$-curve, is of real codimension $2(k-1)$.  
(See Appendix for the proof in the case that $k=2$.  
The argument can be generalized for general $k > 2$.) 
Note also that $\C P^2 \# \overline{\C P^2}$ does not 
contain any cycle of self-intersection number $-2$.   
In our case, $\{ J_{\widetilde{E}_a}^t \}$ gives a three parameter 
family of compatible 
almost complex structures on $\C P^2 \# \overline{\C P^2}$.  
Thus we can take $\{J_{\widetilde{E}_a}^t\}_{0 \leq t \leq 1}$ 
such that there are no $(-k)$-curves, $k \geq 2$,  
in the fibers of $(\widetilde{E}_a,J_{\widetilde{E}_a}^t)$ over $\C P^1$.  
Then, for each $t \in [0,1]$ and $b \in S^2$, there is a 
unique $(-1)$ embedded curve in each fiber 
$(\widetilde{E}_a\vert_b \cong X^4_2, \om, J^t)$.  
Then the process of blowing-down of all 
$J_{\widetilde{E}_a}^t$-holomorphic 
$(A_1 - A_2)$-spheres in the fibers of 
$\widetilde{E}_a$  is isotopic each other as a fiberwise compatible almost 
complex structures as $J^t_{\widetilde{E}_a}$ vary. 
Hence follows that our symplectic fiber bundle 
$(E_a,\Om_K)$ is also trivial, which is a contradiction. \QED
\end{proof}

We shall improve Theorem \ref{Theorem 4.4} in the following statement\footnote{
This proof was suggested by Professor A. Kono, when K.O. gave 
a  proof of this result at his seminar in 1996.}.  

\begin{theorem}\label{Theorem 4.7}
a) The rank of the homomorphism 
$\pi_1 ({\rm Symp}  (M_1,\Om_1)) \to \pi_1({\rm Diff} (M_1))$ 
is at least 1.

b) The rank of the homomorphism $\pi_3 ({\rm Symp} (M_1, \Om_1)) 
\to \pi_3({\rm Diff}(M_1)$
is at least 2.
\end{theorem}

\begin{proof}
a) It is enough to show that the symplectic fiber bundle 
constructed by the loop $(g_t)^k$, $k \neq 0$, in the proof of Theorem \ref{Theorem 4.4}.a 
is a non-trivial
differentiable fibration for all $k$.  
To do so, it suffices to
compute the cohomology ring of this differentiable bundle.  
Since the loop $g_t$ by our choice is the product of the identity element 
in ${\rm Symp}  (N, \om_0)$, our cohomology is also the tensor product 
of the corresponding rings.  
In short we need only to prove the non-triviality of the cohomology ring in the case 
$N = pt$.  

Here, we consider the projectivization $E$ of the bundle $\Vv$ over $\C P^1 \times \C P^1$, which is given in the proof of Theorem \ref{Theorem 4.3}.  
We compute the cohomology ring of $E$.  
We regard $E$ as a family of $\C P^1$-bundles on $\C P^1$ parametrized by 
$\C P^1$, namely a $\C P^1$-bundle over $S^2 \times S^2$.

Let us compute $H^*(E,\Z)$ by using the Leray-Hirsch Theorem.
To see that the first Chern class $c_1(\Vv )$ vanishes, 
it suffices to compute its restriction to $S^2\times \{ pt\}$
and $\{ pt\}\times S^2$.  
The second Chern class $c_2(\Vv )$ equals the Poincare dual of the class of 
the zero section of $\sigma_1$, which extends to a section on $\C P^1 \times \C P^1$.  
This zero locus consists of the only point $(z=\infty, a=0)$.  
So we find that $c_2(\Vv )$ is the generator $ \{ S^2\times  S^2\}$ 
of $H^4(S^2 \times S^2;\Z)$.  
Applying the Leray-Hirsch-Theorem we get 
$$ H^*(E, \Z) = \frac{H^*(S^2\times S^2, \Z)[t]}{t^2-\{ S^2\times S^2\}}.$$
In this ring $t^2$ cannot be devided by 2. But in the ring $H^*(S^2\times S^2\times S^2,\Z)$ any element
$t^2$ can be divided by 2. Hence $E$ is not homotopic to $S^2\times S^2\times S^2$.
That proves the non-triviality of the loop $\{ g_t\}$ in the diffeomorphism group.

To prove that the loop $\{ g_t^n\}$ also realizes a non-trivial element in the diffeomorphism group
we proceed similarly. Namely the loop $\{ g_t^n\}$ corresponds to the $n$-time 
fiber connected sum $E^{\langle n \rangle}$.  
We have the following formula
$$H^*(E^{\langle n \rangle}, \Z) = 
\frac{H^*(S^2\times S^2,\Z)[t]}{t^2- n \{ S^2\times S^2\} }.$$
Now we note that the set of element $x\in H^*(E^{\langle n \rangle}, \Z)$ 
such that $x^2=0$ is the union $\Z (\{S^2\}\times 1) \cup \Z (1\times \{S^2\})$.  
In particular from any 3 elements in this set we can 
get 2 linearly dependent elements.  
This implies that $E^{\langle n \rangle}$ and $S^2\times S^2\times S^2$ 
not homotopic, because there are $3$ linearly independent elements of the last ring,  
namely 
the generators $\{ S^2\}\times 1\times 1$, $1\times \{ S^2\}\times 1$, 
$1\times 1\times \{ S^2\} $, whose square vanish.

b)  It suffices to show that the two subgroups $SO(3) \times Id 
\subset {\rm Symp} (M_1,\om)$ 
and $Id \times SO(3) \subset {\rm Symp}(M_1,\om)$ realize two linearly independent
elements in $\pi_3({\rm Diff}(M_1))$.   
We denote by $E_1$ and $E_2$ two differentiable bundles with fiber $M_1$ over $S^4$, 
which correspond to the elements in $\pi_3({\rm Diff}(M_1))$ realized 
by these subgroups $SO(3)$.   
Let us consider the homotopy exact sequences of these fibration $E_i$, 
which give us two connecting homomorphisms $h_i: \pi_4(S^4)\to \pi_3(M_1)$.  
We observe that $\pi_4(S^4)= \Z$, $\pi_3(M_1) = \Z \oplus \Z \oplus \pi_3(N)$.  
Now it is easy to check that the homomorphism $h_i$ are linearly independent, 
and hence the image of the two subgroups are also linearly independent
elements in $\pi_3(E)$.   
\end{proof}

We can also describe parametrized Gromov-Witten invariants 
in a different way using the Poincar\'e duality.  
Define 

\begin{align}
CGW^E_{g,m,s_A}:& H^*(C\Mm_{g,m} , \R)\to H^{*+\nu}(E^{(m)}, \R) 
\nonumber\\
& \delta \mapsto PD(\delta\backslash \Pi_* (C\Mm_{g,m,s_A}(E, J(E))).  
\end{align}
Here 
$\nu = m \dim M - 2 \langle c_1(M), A \rangle + (\dim M -6) (g-1) -2m$.  

Assume that the local system $\Hh_2(E)$ for the symplectic fiber bundle 
$p:E \to B$ is trivial.  
Using the moduli space with $m=0$, we define the following homomorphism 

\begin{align}
CGW_{g,A}(E) :& H^*(C\Mm_{g,m},\R)\to H^{*+\nu}(B,\R) \nonumber \\
& \delta \mapsto PD (p_*(\delta \backslash \Pi_* (C\Mm_{g,0,s_A}(E, J(E)))). \nonumber
\end{align}

We call the image of the map $CGW_{g,A}(E)$ 
the Gromov-Witten characteristic classes.  

\begin{theorem}
The Gromov-Witten characteristic classes 
are invariants of symplectic fiber bundles $E$ with simple local system 
$\Hh_2(E)$.  
All the Gromov-Witten characteristic classes with positive degree 
vanish for trivial symplectic fiber bundles. 
\end{theorem}

\begin{proof}
The first statement is obvious (cf. Theorem \ref{Theorem 3.3}).  
To prove 
the second statement we compute the Gromov-Witten characteristic classes
  on a trivial bundle $B \times M$. Let $[T]$ be a cycle in $B$. As
before we denote also by $p$ the projection $E^{(m)}\to B$.
\begin{align}
\langle CGW_{g,A}(E)(\delta), [T] \rangle 
&= \langle PD (p_*( \delta\backslash \Pi_*(B \times C\Mm_{g,0}(J,A))), ([T]) \rangle \nonumber \\
&= \langle PD ([B]), [T] \rangle \langle \delta, \Pi_*( C\Mm_{g,0}(J,A)) \rangle .  \nonumber 
\end{align}
Clearly $\langle PD ([B]), [T] \rangle =0$, if $\dim [T]\ge 1$.
\end{proof} 

The word ``characteristic class'' is explained by the following 
functoriality of these classes.  
In particular we see that the Gromov-Witten characteristic classes 
are cohomology classes of the classifying space ${\rm BSymp}_0 (M,\om)$.  
(Note that ${\rm BSymp}_0 (M,\om)$ is simply connected, hence any local systems on it 
are simple.  If we consider the moduli space of vertical stable maps with 
a fixed energy and a fixed Chern number, we can obtain characteristic classes for 
${\rm Symp} (M,\om )$-bundles.)  

\begin{theorem}
Let $E$ be a symplectic fiber bundle over $B$ with the simple local system 
$\Hh_2(E)$ and 
$f^*(E)$ be the induced symplectic fiber bundle 
by a map $f: B_1 \to B$.  
Then we find that 
$$
\widetilde{f}^*(CGW^E_{g,m,A}) = CGW^{f^*E}_{g,m,A}, 
$$
where $\widetilde{f}:f^*E \to E$ is the tautological bundle map of 
symplectic fiber bundles and 
$$CGW_{g,A}(f^*E) = f^* \circ (CGW_{g,A})(E) .$$
\end{theorem}

\begin{proof}
By the construction, we have 
\begin{equation}
C\Mm_{g,m,s_A}(f^*(E)) = (C\Mm_{g,m,s_A}(E)) {}_p \times_f B_1, 
\nonumber \end{equation}
where we take the fiber product in the sense of spaces 
with Kuranishi structures and $p$ is the projection from 
$C\Mm_{g,m,s_A}(E))$ to $B$.  
Hence, we find that 
\begin{equation}
\widetilde{f}^* (PD (\Pi_*(C\Mm_{g,m,s_A}(E))) = 
PD(\Pi_*(C\Mm_{g,m,s_A}(f^*E))). 
\nonumber 
\end{equation}
Hence we can see immediately that
$$
(id. \otimes (\widetilde{f}^*)^{\otimes m})(CGW^E_{g,m,A}) 
= CGW^{f^*E}_{g,m,A}.
$$
When $m=0$, it implies that 
$$CGW_{g,A}(f^*E)= f^* \circ CGW_{g,A}(E)$$
\end{proof}

It is an easy exercise to interpret Theorem \ref{Theorem 4.3} and 
Theorem \ref{Theorem 4.4} 
in term of Gromov-Witten characteristic classes.

\section{Appendix. An alternative proof of  Theorem \ref{Theorem 4.3}}

Let $\om^{(1)}$, $\om^{(2)}$ be symplectic forms on $S^2$ such that 
$\int_{S^2} \om^{(1)} > \int_{S^2} \om^{(2)}$.  
According to Propostion 4.1 it suffices to find a symplectic 
bundle $E$ over $S^2$ with fiber 
$(S^2\times S^2, \om = \om^{(1)} \oplus \om^{(2)})$ and 
a parameterized  Gromov-Witten invariant whose value on $E$ is 
non-trivial.  We shall construct the bundle $E$ by finding its 
transition function $g$, i.e., a loop in 
${\rm Symp} (S^2\times S^2,\om)$.  
The existence of such element 
$g$ was shown by Gromov \cite{Gromov1985}, 
and in what follows we shall give 
a detailed proof.  First we need the following lemma
(compare with \cite{Gromov1985}, $2.4.C$).
Denote by $A_1$, resp. $A_2$ the homology classes 
$[S^2 \times \{ pt \}]$, resp. $[\{ pt \} \times S^2]$.  

\begin{lemma}\label{Lemma 5.1}
The subspace $\Jj_0$ of compatible almost complex 
structures $J$ on $S^2\times S^2$, 
for which there exists a $J$-holomorphic 
sphere in  a class $A_1 - \ell A_2$ for some $\ell \geq 1$, 
is a non-empty closed subset of codimension 2 
in $\Jj(S^2\times S^2, \om)$.  
For each $J\in (\Jj(S^2\times S^2,\om)\setminus \Jj_0)$ and 
for each point in $S^2\times S^2$ there is a unique $J$-holomorphic 
sphere representing class $A_i$ and passing through $x$.  
Moreover, these spheres are embedded.  
\end{lemma}

\begin{proof}
First of all, we note that for any $J \in \Jj(S^2\times S^2,\om)$ 
there exists a unique embedded $J$-holomorphic sphere 
representing the class $A_2$ and passing through each point $x$, 
cf. \cite{McDuff1991}.  
We include here the proof of this fact for the reader's convenience.  

For a generic compatible almost complex structure $J$, there exists 
such a $J$-holomorphic sphere.  
For $J_{\infty} \in \Jj(S^2\times S^2,\om)$, pick a sequence $\{J_i \}$ 
of generic compatible almost complex structures, which converges to 
$J_{\infty}$.  
Let $u_i$ be the $J_i$-holomorphic sphere representing the class $A_2$ 
and passing through $x$.  
Suppose that there exists a subsequence $u_{i_k}$ 
converging to a $J_{\infty}$-holomorphic sphere $u_{\infty}$.  
Since $A_2$ is a primitive class, the adjunction formula implies that 
$u_{\infty}$ is embedded.  Clearly it passes through $x$.  
Hence we obtain the desired existence.  
If it is not the case, a subsequence of $\{u_i \}$ converges 
to a $J_{\infty}$-stable map and there appears a $J_{\infty}$-holomorphic 
map $v$ representing the class $kA_2 - \ell A_1$ for some integers 
$k$ and $\ell$ such that $k$ is positive and $(k, \ell) \neq (1,0)$.  
If $v$ is multiply covered, factorize it as $v = p \circ v'$, where 
$v'$ is a simple map and $p$ is a ramified covering of $\C P^1$.  
Replace $v$ by $v'$, if necessary, we may assume that $v$ is simple.  
Since $\int_{A_1} \om > \int_{A_2} \om$ and 
$v$ is a $J_{\infty}$-holomorphic map with the symplectic area 
smaller than $\int_{A_2} \om$, we have 
$k \geq \ell + 1$ and $\ell \geq 1$.  
By the adjunction formula, the virtual genus of $C=v(\C P^1)$ is 
\begin{align}
g_v(C) & = 1 + \frac{1}{2} (C \cdot C - c_1(C) ) \nonumber \\
       & = 1 - k \ell - k + \ell \nonumber \\
       & \leq  1 - (\ell + 1) \ell - (\ell + 1) + \ell \nonumber \\
       & < 0. \nonumber
\end{align}
However, the virtual genus $g_v(C)$ is a non-negative integer, which 
is a contradiction.  
Therefore we obtain the existence of $J_{\infty}$-holomorphic 
sphere representing the class $A_2$ and passing through the given point 
$x$.  

Next we prove that $\Jj_0$ is a non-empty closed subset of 
codimension 2.  
The subspace $\Jj_0$ is non-empty because the sphere 
$(x,-x)$ is symplectic.  
To prove the closedness of $\Jj_0$ we first notice 
that the energy of a holomorphic sphere 
in class $A_1 - \ell A_2$ is less than $\om^{(1)}(A_1)$. 
Thus we can apply the Gromov compactness argument 
to the following situation.  
Let a sequence of $J_i$-holomorphic spheres $u_i$ 
representing $A_1 - \ell A_2$.  
Suppose that $J_i$ converges to a compatible almost complex structure 
$J_{\infty}$.   
Then there is a subsequence $\{ u_{i_k} \}$, which converges to 
a $J_{\infty}$-stable map $u_{\infty}$.  
If $u_{\infty}$ is a $J_{\infty}$-holomorphic sphere, we find that 
$J_{\infty} \in \Jj_0$.  (In this case, by the adjunction formula, 
$u_{\infty}$ is an embedding.)  
Otherwise, $u_{\infty}$ consists of at least two irreducible components, 
which represent the classes 
$k_i A_1 - \ell_i A_2$  such that $\sum k_i= 1$
and $\sum \ell_i =\ell$.  
As we mentioned above, there is always a $J_{\infty}$-holomorphic sphere 
in class $A_2$.  
Taking into account of possivity of intesection 
in dimension 4, we conclude that these $J_{\infty}$-holomorphic spheres 
must be of type 
$A_1 - \ell_i A_2$ and $m_jA_2$ such that 
$\sum (-\ell_i) + \sum m_j =-\ell$.  
Since $m_j$, if exists,  must be positive, 
we conclude that there must be a bubble of
type $A_1 - \ell'A_2$, $\ell' \ge 1$, that proves the
closedness of $\Jj_0$. 

The codimension of $\Jj_0$ is at least 2 
by a similar argument as in 
\cite{LO1996}, \cite{LO2001}.  
(Namely for a fixed $\ell$ we consider the universal moduli 
space of the pairs $(J, J\mbox{\rm -holomorphic \  sphere \ in \ class}~ A_1-lA_2)$.  
The Fredholm index of the projection of this moduli space 
on the first factor is equal $4 + 2(1-\ell ) -6 \le -2$.)  
To prove that the codimension is precisely 2, we use the uniqueness 
of $J$-holomorphic sphere in class $A_1-A_2$ if it exists. (cf. with the 
argument in \cite{HLS}. It follows that
the kernel of the linearization of the projection from the universal 
moduli space to the first factor, i.e., the space of compatible 
almost complex structures equals zero.).

Finally we prove the existence of $J$-holomorphic sphere in class
$A_1$ for $J\in (\Jj\setminus \Jj_0)$ and use again the bubbling-off 
argument.  
For a generic compatible almost complex structure $J$, there exist 
$J$-holomorphic spheres in the class $A_1$, see \cite{McDuff1991}.  
Note that such $J$-holomorphic spheres are automatically embedded 
by the adjunction formula and that the class $A_1$ is primitive.  
If there is no $J_{0}$-holomorphic curve in the class $A_1$, 
we pick a sequence of generic compatible almost complex structures 
converging to $J_{0}$.  
Then the bubbling-off argument implies that 
there must be a $J_0$-holomorphic sphere 
in a class $A_1 -\ell  A_2$ for $\ell \geq 1$.  
\end{proof}

Now let us find an element $g\in {\rm Symp} (S^2\times S^2,\om^{(1)}\oplus \om^{(2)})$
by studying the action of the group of symplectomorphisms on
$\Jj\setminus \Jj_0$.  
Since $\Jj_0$ is a closed subset of codimension 2 we can choose a small 
disk $D$ in $\Jj(S^2\times S^2)$ such 
that this disk intersects $\Jj_0$ transversally 
at exactly one interiour point.
By results of Gromov \cite{Gromov1985} and McDuff \cite{McDuff1991}, 
for any compatible almost complex structure $J_{\theta}$,  
$\theta \in \partial D$, $S^2 \times S^2$ is foliated by 
$A_1$-curves and $A_2$-curves, respectively.  
In particular, $J_{\theta}$ is pointwisely positive on these $A_1$-curves 
and $A_2$-curves.  
It implies that there is a loop $g_t$ in the group 
${\rm Symp}  (S^2\times S^2, x)$ such that the image 
$g_t(J_0)$ is homotopic to the loop $\p D$ in $\Jj \setminus \Jj_0$, 
where $J_0$ is the complex structure on $\C P^1 \times \C P^1$.  
Thus, we can deform $D$ along the boundary so that 
$D$ intersects $\Jj_0$ transversally at one point and 
$\partial D = \{g_t (J_0) \}$.   

Now we construct our bundle $E$ by gluing two trivial
$S^2\times S^2$ bundles over another disk $D'$ using this loop $g_t$.  
Since the base space of $E$ is $S^2$, which is simply connected, 
$\Hh_2(E)$ is a simple local system.  
We claim that the parametrized Gromov-Witten invariant 
$I^E_{0,0,A_1-A_2}$ is 1. 
Since $c_1(A_1-A_2) =0$, the moduli space 
$C\Mm_{0,0}(E,J(E),A_1-A_2)$ is $0$-dimensional.  
To compute the invariant for our bundle $E$, we choose 
a generic fiberwise compatible
almost complex structur $J(E)$ on $E$ as follows.   
Note that it is the case of weakly monotone symplectic manifolds 
and we are working with Gromov-Witten invarians of genus 0, 
$c_1(A_1 -A_2) = 0$ and the dimension of the base of $E$ is $2$.  
Therefore it suffices to perturb $J$ to get the fundamental class
of the corresponding moduli space.  
We observe that the standard product complex structure $J_0$ on $S^2 \times S^2$ is 
 $(A_1-A_2)$-regular.  
Then we take
$J(E)$ being the gluing of the constant complex structure  $J_0$ 
over $D'$ and
the compatible almost complex structure parametrized by $D$ 
along the boundary $\p D$ by $g_t$.  
By the transversality of the intersection of $J(E)$ with $\Jj_0$, 
we find that
the vertical almost complex structure $J(E)$ is $(A_1-A_2)$-regular 
for the symplectic fiber bundle.
By the construction the parametrized moduli space
$C\Mm_{0,0, A_1-A_2}(E, J(E))$ consists of one point over the point 
$D \cap \Jj_0$.  
Therefore the value $I^E_{0,0,A_1-A_2} = 1$.
By Proposition 4.1, this nontrivial parametrized Gromov-Witten invariant
defines a non-trivial element in $Hom (\pi_2({\rm BSymp}  (M,\om), \Q)$.

{\rm H\^ong-V\^an L\^e},
Mathematical Institute, Zitna 25,
CZ-11567 Praha, Czech Republic \\
{\rm Kaoru Ono},
Department of Mathematics,
Hokkaido University,
Sapporo, 060-0810, Japan

\end{document}